%% file: 20250521.tex
\documentclass[11pt,letterpaper]{amsart}
\input{header}

\makeatletter
\@namedef{subjclassname@2020}{\textup{2020} Mathematics Subject Classification}
\makeatother

\begin{document}

\title{Rigidity of CMC hypersurfaces in 5-and 6-manifolds}

\author{Han Hong}
\address{Department of Mathematics and statistics \\ Beijing Jiaotong University \\ Beijing \\ China, 100044}
\email{hanhong@bjtu.edu.cn}

\author{Zetian Yan}
\address{Department of Mathematics \\ UC Santa Barbara \\ Santa Barbara \\ CA 93106 \\ USA}
\email{ztyan@ucsb.edu}
\begin{abstract}
 We prove that nonnegative $3$-intermediate Ricci curvature combined with uniformly positive $k$-triRic curvature implies rigidity of complete noncompact two-sided stable minimal hypersurfaces in a Riemannian manifold $(X^5,g)$ with bounded geometry. The stonger assumption of nonnegative $3$-intermediate Ricci curvature can be replaced by the nonnegativity of Ricci and biRic curvature. In particular, there is no complete noncompact stable minimal hypersurface in a closed $5$-dimensional manifold with positive sectional curvature. This extends result of Chodosh-Li-Stryker [J. Eur. Math. Soc (2025)] to $5$-dimension. We also establish rigidity results on CMC hypersurfaces with nonzero mean curvature in $5$- and $6$-manifolds.
\end{abstract}
\maketitle

\section{Introduction}
An immersed complete two-sided  stable minimal hypersurface $M^n$ in orientable manifold $(X^{n+1},g)$ has zero mean curvature and satisfies the inequality
\begin{equation}\label{stableinequality}
    \int_M \left(|A_M|^2+\Ric^X(\eta,\eta)\right)\psi^2\leq \int_M |\nabla \psi|^2
\end{equation}
for any $\psi\in C^{\infty}_c(M)$, where $A_M$ is the second fundamental form of the immersion, $\eta$ is a unit normal vector on $M$ and $\Ric^X$ is the ambient Ricci curvature tensor. In a similar manner of stable geodesic in $2$-dimensions, stable minimal hypersurface is a powerful tool to study the ambient geometry of $(X^{n+1},g)$ in higher dimensions. Assuming $M$ is compact and orientable, standard arguments show that
\begin{enumerate}
    \item When $\Ric^X>0$, there are no closed two-sided stable minimal hypersurfaces. If $\Ric^X\geq 0$, then $M$ is totally geodesic and $\Ric^X(\eta,\eta)\equiv 0$ along $M$; see \cite{Simons-minimal-varieties}.
    \item If the ambient scalar curvature is strictly positive, then $M$ also admits a metric of positive  scalar curvature \cite{Schoen-Yau-incompressible-minimal-surfaces, Schoen-Yau-PSC}. In three dimension, each component of $M$ must be a $2$-sphere.
\end{enumerate}

Now, if one assumes that $M^n$ is non-compact with respect to the induced metric, it becomes very complicated and subtle. The theory has been well-developed in $3$-dimensions: if $M^2\hookrightarrow (X^3,g)$ is a complete noncompact two-sided stable minimal immersion, it holds that
\begin{enumerate}
    \item When $R_X\geq 0$, $(M,g)$ is conformal to either a plane or a cylinder \cite{Fischer-Colbrie-Schoen-The-structure-of-complete-stable}. In the latter case, $M$ is totally geodesic, intrinsically flat and $R_X\equiv 0$ along $M$. 
    \item When $\Ric^X\geq 0$, $(M,g)$ is totally geodesic, intrinsically flat and has $\Ric^X(\eta,\eta)\equiv 0$ along $M$ \cite{Schoen-Yau-positive-Ricci-curvature}. Moreover, either $X^3$ is diffeomorphic to $\bR^3$ or its universal cover is isometric to the Riemann product $M^2\times \bR$ \cite{Liu-nonnegative-Ricci-curvature}.
    \item When $R_X\geq 1$, $M$ must be compact \cite{Schoen-Yau-black-hole}. 
\end{enumerate}

Recently, Chodosh, Li and Stryker developed such theory in $4$-dimensions by combining nonnegativity of $2$-intermediate Ricci curvature with strict positivity of scalar curvature. 

\begin{theorem}[\cite{chodosh-li-stryker}]
 If $(X^4,g)$ is a complete $4$-dimensional Riemannian manifold with weakly bounded geometry, satisfying
    \[\operatorname{Ric}^X_2\geq 0\  \text{and}\ R_X\geq \delta>0, \]
where $\Ric_2^X$ is the $2$-intermediate Ricci curvature (Definition \ref{n-1intermediate}). Then any complete two-sided stable minimal immersion $M^3\rightarrow (X^4,g)$ is totally geodesic and has $\operatorname{Ric}^X(\eta,\eta)\equiv 0$ along $M$.
    
\end{theorem}

It is natural to ask what happens in higher dimensions. A fundamental question is to determine which ambient curvature conditions ensure compactness or rule out the existence of complete stable minimal hypersurfaces. As discussed above in $3$-manifolds, $R_X\geq 1$ implies compactness and $\Ric^X>0$ implies non-existence. However, two examples constructed in \cite{chodosh-li-stryker}[Example 1.1,1.2] can be easily adapted to show that neither of these results can hold in higher dimensions. 

Note that for most of aforementioned results, it is essential that no assumption is made on properness or volume growth of $M$. The result in \cite{chodosh-li-stryker} cannot extend to $5$-dimensions directly because $\Ric^X_2\geq 0$ and $R_X>\delta$ are tailor to obtain one-endness and volume growth of hypersurfaces in $4$-dimensions. Hence, exploring alternative curvature conditions to maintain one-endness and volume growth in higher dimensions is the main topic in this paper.

\subsection{Main results}
The purpose of this article is to extend results in \cite{chodosh-li-stryker} to $5$ and $6$ dimensions under different curvature assumptions and stability of CMC hypersurfaces. 

Before discussing our main results, we recall some definitions and introduce some notions.

\begin{definition}\label{weakstabilitydefinition}
    Let $(X^{n+1},g)$, $n\geq 2$, be a  Riemannian manifold. An immersed CMC hypersurface (two-sided if it is minimal) $M^n$ in $X$ is  stable if $$\int_M \left(|A_M|^2+\Ric^X(\eta,\eta)\right)\psi^2\leq \int_M |\nabla \psi|^2$$
    for $\psi\in C_c^\infty(M)$. 
    
\end{definition}

\begin{definition}\label{triRic}
    Let $(X^{n+1},g)$, $n\geq 2$, be a  Riemannian manifold. Given a local orthonormal basis $\{e_i\}_{i=1}^{n+1}$ of $T_pX$ and $k, \alpha>0$, we define \textit{$(k,\alpha)$-triRicci curvature} in the direction $(e_i;e_j,e_l)$ as
    \begin{equation}
        (k,\alpha)\operatorname{-triRic}_X(e_i;e_j,e_l):=k\Ric^X(e_i,e_i)+\alpha\text{-}
        \biRic_{e_i^{\perp}}(e_j,e_l),
    \end{equation}
    where $\alpha\text{-}\biRic_{e_i^{\perp}}(e_j,e_l)$ is given by
    \begin{equation*}
        \alpha\text{-}\biRic_{e_i^{\perp}}(e_j,e_l):=\sum_{m\neq i}\R^X_{mjmj}+\alpha\sum_{m\neq i}\R^X_{mlml}-\alpha\R^X_{jljl}, \quad j\neq l
    \end{equation*}
    i.e., we take the partial contraction on the subspace $e_i^{\perp}:=\vspan \{e_l\}_{l\neq i}$ normal to $e_i$. Furthermore, we define $$(k,\alpha)\operatorname{-triRic}_X(p)=\inf \{(k,\alpha)\operatorname{-triRic}_X(e_i;e_j,e_l)\},$$
    where the infimum is taken over all local orthonormal basis $\{e_i\}_{i=1}^{n+1}$ at $p$, and over all admissible index triples $(e_i;e_j,e_l)$, $i\neq j\neq l$.
    If $\alpha=1$, we simply denote it by $k\operatorname{-triRic}_X$.
\end{definition}

\begin{remark}
    When $n=2$ or $3$, $1\operatorname{-triRic}_X$ is just the half of the scalar curvature of $X$. In higher dimensions, $1\operatorname{-triRic}_X$ is the $3$-intermediate curvature $\mathcal{C}_3$ defined by Brendle-Hirsch-Johne \cite{brendlegeroch'sconjecture}.
\end{remark}
\begin{remark}\label{triRicandscalar}
    Direct calculation yields
    \begin{equation*}
    \begin{split}
        &\sum_{i=1}^{n+1}\sum_{j\neq i, l\neq i} k\operatorname{-triRic}_X(e_i;e_j,e_l)\\
        &=\sum_{i=1}^{n+1} \left(kn(n-1)\Ric^X(e_i,e_i)+(2n-3)\sum_{j\neq i, l\neq i}\R^X_{jljl}\right)\\
        &=kn(n-1)\sum_{i=1}^{n+1}\Ric^X(e_i,e_i)+(2n-3)(n-1)\sum_{j,l=1}^{n+1}\R^X_{jljl}\\
        &=((k+2)n-3)(n-1)R^X. 
        \end{split}
    \end{equation*}
In particular, positive $k$-triRicci curvature implies that $X$ has positive scalar curvature. Conversely, $\Sigma^2\times S^{3}(5/12)$ has positive scalar curvature but does not have positive $1$- triRicci curvature where $\Sigma^2$ is a closed hyperbolic surface and $S^3(5/12)$ is a sphere with constant curvature $5/12$.
\end{remark}
\begin{remark}
    When $(X^{n+1},g)$ is the hyperbolic space, $k\operatorname{-triRic}_X=3-(k+2)n.$
\end{remark}

\begin{definition}\label{n-1intermediate}
    Let $(X^{n+1},g)$, $n\geq 2$, be a Riemannian manifold. The $(n-1)$-intermediate Ricci curvature of $X$ at $p\in X$ in the direction $e_l$,  is defined as
    \begin{equation*}
\Ric^X_{n-1}(e_i;e_l)=\sum_{j\in\{1,\cdots,n\}\setminus i} \R_X(e_j,e_l,e_j,e_l)
    \end{equation*}
    where $\{e_1,\cdots,e_n\}$ is any orthonormal basis of $T_pX$. Similarly, we define $$\Ric^X_{n-1}(p)=\inf\{\Ric^X_{n-1}(e_i;e_l): \ i\neq l\}.$$
\end{definition}


We now state our results on minimal hypersurfaces in 5-manifolds; more general results can be found in Section \ref{proof-of-main-results}.

\begin{theorem}\label{theorem1introduction}
    If $(X^5,g)$ is a complete $5$-dimensional Riemannian manifold with weakly bounded geometry, satisfying
    \[\operatorname{Ric}^X_3\geq 0\  \text{and} \ k\operatorname{-triRic}_X\geq \delta>0 \ \text{for}\ k\in[1,2]. \]
Then any complete two-sided  stable minimal immersion $M^4\hookrightarrow X^5$ is parabolic with at most two ends. In particular, it is totally geodesic and has $\operatorname{Ric}^X(\eta,\eta)\equiv 0$ along $M.$
\end{theorem}

 It is possible to weaken non-negativity of $\Ric_3^X$ to a mixed curvature condition. Note that $\Ric_3^X\geq 0$ implies $\Ric^X\geq 0$ and $\biRic_X\geq 0$, where $\biRic_X$ is the bi-Ricci curvature (see Definition \ref{biRicci} in detail). In the Appendix, we provide a (one parameter family) closed $5$-dimensional Riemannian manifold $(X^5,g)$, satisfying
    \[\operatorname{Ric}^X> 0,\ \, \biRic_X> 0,\ \, \text{and}\ \ 1\operatorname{-triRic}_X\geq \delta>0, \] while $\Ric^X_3<0$ somewhere. As a matter of fact, the result in Theorem \ref{theorem1introduction} holds with $\Ric_3^X\geq 0$ replaced by $\Ric^X\geq 0$ plus $\biRic_X\geq 0$.

\begin{theorem}\label{maintheorem1introduction}
    If $(X^5,g)$ is a complete $5$-dimensional Riemannian manifold with bounded geometry, satisfying
    \[\operatorname{Ric}^X\geq 0,\ \, \biRic_X\geq 0,\ \, \text{and}\ \ k\operatorname{-triRic}_X\geq \delta>0\ \text{for}\ k\in[1,2]. \]
Then any complete two-sided  stable minimal immersion $M^4\rightarrow X^5$ is parabolic with at most two ends. In particular, it is totally geodesic and has $\operatorname{Ric}^X(\eta,\eta)\equiv 0$ along $M.$
\end{theorem}

Let us make some comments here. The first one is that if $(X^5,g)$ is a manifold with bounded geometry satisfying $\Ric^X>0$, $\operatorname{biRic}_X\geq 0$ and $k\operatorname{-triRic}_X\geq \delta>0\ \text{for}\ k\in[1,2]$, then there is no complete two-sided stable minimal immersion. This follows directly from Theorem \ref{maintheorem1introduction}. As a consequence, there is no two-sided stable minimal immersion in $(\mathbb{S}^5,g_{can})$ (previously proved in \cite{catino}). The second one is that there is no complete two-sided stable minimal immersion in manifold $X$ with $\Ric^X>0 $ and $\biRic_X\geq \delta >0$. Indeed, uniformly positive $\biRic_X$ implies that a complete two-sided stable minimal immersion must be compact \cite{shenyingyerugang}. Then the standard argument leads to contradiction.  Finally, we would like to mention that the conclusion of at most two ends in above two theorems can be compared to the result proved by Fischer-Colbrie and Schoen \cite{Fischer-Colbrie-Schoen-The-structure-of-complete-stable} that a complete two-sided stable minimal surface in a manifold with nonnegative scalar curvature has at most two ends (plane or cylinder). Furthermore, we prove some related results in dimension $n+1=6$ (see Section \ref{n=5}).

\subsection{Main ideas}
The strategy in this article follows from that in \cite{chodosh-li-stryker}. A common feature in the proofs of Theorem \ref{theorem1introduction} and Theorem \ref{maintheorem1introduction} is the estimate on volume growth of ends by using $\mu$-bubbles, a technique introduced by Gromov \cite{gromovmububble}. The method in \cite{chodosh-li-stryker} cannot extend to $5$-dimensions directly because one cannot control curvature terms from the second fundamental form, the ambient Ricci curvature and the positive Jacobi function in the second variation of the warped $\mu$-bubble constructed there. The novelty of our method is to introduce a parameter (the exponent of the warped function) which brings more freedom to exploit curvature conditions. The same strategy works well in the resolution of stable Bernstein problem in $\mathbb{R}^6$ \cite{mazet}.

Another interesting feature of Theorem \ref{theorem1introduction} and Theorem \ref{maintheorem1introduction} is the combination of different curvature conditions which enter the proof at different places. For completeness, we present the overall strategy as follows.

Let $M^n\hookrightarrow (X^{n+1},g)$ $(4\leq n\leq 5)$ be a complete two-sided stable minimal hypersurface.
\begin{enumerate}
    \item To show the vanishing of $A_M$ and $\operatorname{Ric}^X(\eta,\eta)$, we mainly prove that the hypersurface is parabolic. Then the standard equation $\Delta u+(|A_M|^2+\Ric^X(\eta,\eta)) u=0$ will imply the result.

    \item In dimension $6$, the fact that $M$ has at most one nonparabolic end follows from the Schoen--Yau inequality which requires only $\biRic_X\geq 0$ together with the infinite volume of $M$ , ensured by the bounded geometry assumption. When $\Ric^X_3\geq 0$, bounded geometry condition can be weaken to weakly bounded geometry; see \cite{chodosh-li-stryker}[Theorem 4.1] for details. In contrast, in dimension $5$, the splitting result in Theorem \ref{oneend1} implies that $M^4$ has at most two ends. Moreover, if $M^4$ has two ends, it must be a product manifold, see \cite{antonelli-pozzetta-xu,catino-Mari-Mastrolia-Roncoroni,hong-wang-spectral-splitting}. Therefore, we only need to consider the case where $M^4$ has one end.

    \item The construction of a good test function along the nonparabolic end is the main issue in the remaining proof.  This is where the positivity of $k\operatorname{-triRic}$ comes in. By conformal change, it is well known \cite{Schoen-Yau-PSC} that a stable minimal hypersurface admits a metric with positive scalar curvature if the ambient metric does. It is interesting to understand if the positivity of $k\operatorname{-triRic}$ can be inherited by hypersurfaces in some sense.

    As mentioned in \cite{chodosh-li-stryker}, it is difficult to prove the linear volume growth of the nonparabolic end. Similarly, we consider an exhaustion $\{\Omega_l\}_{l=1}^{\infty}$ of $E$ and use the theory of $\mu$-bubbles to show $\partial \Omega_l\cap E$ has controlled volume. This is one difference comparing to \cite{chodosh-li-stryker}. To control curvature terms in the second variation of the $\mu$-bubble, we introduce a parameter $k$ as the exponent of the warped function. Suitable choice of $k$ yields the desired volume estimate; see Lemma \ref{diameterforclosedbubble} in detail. This is the step that needs the assumption of strictly positive $k\operatorname{-triRic}$.  
    \item We need to upgrade volume control of the bubble to volume control of the slice which needs a lower bound on the intrinsic Ricci curvature of $M$. It is achieved by combining several curvature conditions and curvature estiamtes in Lemma \ref{estimate}.
    \item Finally, to prove parabolicity, we need to 
construct test functions $\varphi_i$ such that $\int_M|\nabla\varphi_i|^2$ tends to zero. This is achievable due to the facts that we have at most one nonparabolic end and the nonparabolic end satisfies almost linear volume growth.
    
\end{enumerate}

\subsection{Organization}
In Section 2, we recall some preliminaries on curvature estimates and ends of non-compact manifolds. In Section 3, we establish the splitting theorem and the one-end theorem under various curvature conditions. Section 4 is devoted to the balanced warped $\mu$-bubble exhaustion and volume growth estimates. The main results are proved in Section 5. The Appendix presents an example related to Theorem \ref{maintheorem1introduction}.
\subsection{Acknowledgements}
We would like to thank Prof. Xianzhe Dai, Prof. Martin Li, Prof. Guofang Wei and Prof. Rugang Ye for their interests in this work. The first author is supported by NSFC No. 12401058 and the Talent Fund of Beijing Jiaotong University No. 2024XKRC008. The second author is supported by a Simons Travel Grant.

\section{Preliminary}
\subsection{Curvature conditions}
If $(X^{n+1},g)$ is a Riemannian manifold, we denote the Riemann curvature tensor, the sectional curvature, the Ricci curvature tensor and the scalar curvature by $\R^X_{ijij}$, $\sec_X$, $\Ric^X_{ij}$ and $\R^X$ respectively. Let $M\hookrightarrow X^{n+1}$ be an immersion. We write $\Ric^M$ and $R^M$ to denote the Ricci curvature and scalar curvature of the pullback metric on $M$. We set the mean curvature to be the trace of the second fundamental form of the immersion.
\begin{definition}
	A manifold $X^{n+1}$ has \textit{bounded geometry} if there exist positive real numbers $\delta$ and $\lambda$ such that on $X^{n+1}$
	\begin{equation*}
		|\sec_X|\leqslant \delta^2, \quad \inj_X\geqslant \lambda
	\end{equation*}
 where $\inj_X$ is the injectivity radius.
\end{definition}

\begin{definition}\label{biRicci}
    Let $(X^{n+1},g)$ be a Riemannian manifold and $\{e_i\}_{i=1}^{n+1}$ be an orthonormal basis for $T_pX$, $p\in X$. The \textit{bi-Ricci curvature} along directions $(e_1,e_2)$ is defined by
    \begin{equation*}
        \biRic_X(e_1,e_2):=\sum_{i=2}^{n+1}\R^X(e_1,e_i,e_i,e_1)+\sum_{j=3}^{n+1}\R^X(e_2,e_j,e_j,e_2).
    \end{equation*}
   We denote $\biRic_X:=\inf\{\biRic_X(e_1,e_2):\ p\in X\ \text{and}\ e_1, e_2\in T_pX\}.$
\end{definition}
This concept was used by Shen-Ye \cite{shenyingyerugang} to study diameter of stable minimal submanifolds.
\begin{lemma}\label{Ric-lower}
    Let $(X^{n+1},g)$ be a $n+1$-dimensional Riemannian manifold with bounded geometry. If $M^n\hookrightarrow (X^{n+1},g)$ is an immersed CMC hypersurface, then
    \begin{equation*}
        \Ric^M\geqslant -Cg,
    \end{equation*}
    where the positive constant $C$ depends on $n$, $\sec_X$, $H^M$ and $\sup |A^M|$.
\end{lemma}
\begin{proof}
    Choose an orthonormal basis $\{e_i\}_{i=1}^{n+1}$ of $X^{n+1}$ at a point $p\in M$ with $e_{n+1}$ normal to $M$. The traced Gauss equation gives
    \begin{equation*}
        \Ric^M(e_1,e_1)=\sum_{i=1}^{n} \R^X_{i1i1}+A^M(e_1,e_1)H^M-\sum_{j=1}^{n}A^M(e_1,e_j)^2,
    \end{equation*}
    which yields the desired result.
\end{proof}

Without bounded geometry, there is an alternative way to ensure the lower bound on $\Ric_M$.
\begin{lemma}\label{Ric-lower1}
    Let $(X^{n+1},g)$ be a $(n+1)$-dimensional Riemannian manifold with $\Ric^X_{n-1}\geq C_1$ for a constant $C_1$. If $M^n\hookrightarrow (X^{n+1},g)$ is an immersed CMC hypersurface, then
    \begin{equation*}
        \Ric^M\geqslant -Cg,
    \end{equation*}
    where the positive constant $C$ depends on $n$, $C_1$, $H^M$ and $\sup|A^M|$. In particular, $C$ can be chosen to be zero if $M$ is totally geodesic and $C_1=0$.
\end{lemma}
\begin{proof}
    The result holds because
    \begin{align*}
        \Ric^M(e_1,e_1)=&\sum_{i=1}^{n} \R^X_{i1i1}+A^M(e_1,e_1)H^M-\sum_{j=1}^{n}A^M(e_1,e_j)^2\\
        \geq & \frac{5-n}{4}H_M^2-|A^M|^2
    \end{align*}
    where the inequality follows from $$|A^M|^2+H^M A^M(e_1,e_1)-\sum_{i=1}^nA^M(e_1,e_j)^2\geq \frac{5-n}{4}H_M^2.$$
        
\end{proof}

\subsection{Curvature estimates}
In this subsection, we recall the standard curvature estimates for stable CMC hypersurfaces in non-compact manifolds.


\begin{lemma}\label{estimate}
    Let $(X^{n+1},g)$ $(4\leq n\leq 5)$ be a $(n+1)$-dimensional complete Riemannian manifold with weakly bounded geometry. Then there is a constant $C$ depending on $X$ and $H^M$ such that every complete noncompact stable CMC hypersurface  (two-sided if it is minimal) $M^n\hookrightarrow (X^{n+1},g)$ with the mean curvature $H^M$ and compact boundary (possibly empty) satisfies
    \begin{equation*}
        \sup_{q\in \Omega}|A^M(q)|\min \{1,d_M(q,\partial \Omega)\}\leqslant C
    \end{equation*}
    for any compact subset $\Omega\subset M$.
\end{lemma}

\begin{proof}
    Observe that when scaling the metric by the norm of the second fundamental form that tends to infinity, the limit of sequence of stable CMC hypersurface is a stable minimal hypersurface. Therefore, with the stable Bernstein theorem in $\mathbb{R}^5$ and $\mathbb{R}^6$, the statement follows from \cite{chodosh-li-stryker}[Lemma 2.4] and \cite{Hong24}[Lemma 3.6] with minor modification; see also \cite{Rosenberg-Souam-Toubiana}. 
\end{proof}

\subsection{Ends of manifold}
We now recall some preliminary results on parabolicity and nonparabolicity of manifolds; see \cite{Li-Geometric-analysis} for more details. 
\begin{definition}
    Let $M$ be a complete noncompact manifold, and let $K\subset M$ be a compact subset and let $E$ be an unbounded connected component of $M\setminus K$ with smooth boundary. Then $E$ is said to be nonparabolic with respect to $K$ if there exists a positive harmonic function $f$ satisfying 
    \[f|_{\partial E}=1,\ \text{and}\ \inf_E f=0.\]
    Otherwise, $E$ is parabolic.
\end{definition}
\begin{definition}
    Fix a point $x\in M,$ and fix $L>0.$ A collection of open sets $\{E_k\}$ is an end of $M$ adapted to $x$ and length scale $L$ if each set $E_k$ is an unbounded connected component of $M\setminus \bar{B}_{kL}(x)$ and satisfies $E_{k+1}\subset E_k.$
\end{definition}
Without lost of generality, we usually ignore the fixed point $x$ if it is clear in the context.
\begin{definition}
    A boundary end $\{E_k\}_{k=1}$ with length scale $L>0$ is nonparabolic if for every $k$, $E_k$ is an unbounded component of $M\backslash B_{kL}$ and $E_k$ is nonparabolic.
\end{definition}

\section{Ends of hypersurfaces} 

\subsection{Splitting theorem and finite ends}
In this subsection,  we show that in a manifold $(X^{n+1},g)$, the number of ends of a stable CMC hypersurface $M^n$ is at most two under certain conditions.

The stability of $M$ implies that for any $\varphi\in C_c^1(M)$,
\[\int_M |\nabla\varphi|^2\geq \int_M (|A|^2+\Ric^X(\nu,\nu))\varphi^2.\]
Let $\Ric^M(x)=\inf_{x\in M}\{\Ric^M(\tau,\tau): |\tau|=1\ \text{and}\ \tau\in T_xM\}$. Suppose $e_1$ is an unit tangent vector of $M$ at $x$ such that $\Ric^M(x)=\Ric^M(e_1,e_1)$. Then the curvature rearrangement gives
\begin{align*}
    \int_M |\nabla \varphi|^2+\alpha \Ric^M(x)\varphi^2\geq&\int_M (|A|^2+\Ric^X(\nu,\nu)+\alpha\Ric^M(e_1,e_1))\varphi^2\\
    \geq& \int_M (\Ric^X(\nu,\nu)+\alpha \sum_{i=1}^n \R^X_{1i1i})\varphi^2\\
    +& \int_M (|A|^2+\alpha A_{11}H-\alpha \sum_{i=1}^n A^2_{1i})\varphi^2.
\end{align*}
Terms involving the second fundamental from on the right-hand side is bounded from below by
\[\frac{4-(n-1)\alpha^2}{4((1-\alpha)n+\alpha)}H^2, \quad \alpha<\frac{n}{n-1};\]
see Deng \cite[Lemma 2.1]{dengqintao}. Moreover, the equality holds iff 
\begin{equation*}
    A_{11}=-\frac{(n-1)\alpha-2}{2((1-\alpha)n+\alpha)}H,\quad A_{ii}=\frac{2-\alpha}{2((1-\alpha)n+\alpha)}H, \quad i\neq 1.
\end{equation*}

Assuming that $\alpha\text{-}\biRic_X(e_1,\nu):=\Ric^X(\nu,\nu)+\alpha \sum_{i=1}^n \R^X_{1i1i}$ satisfies
\[\alpha\text{-}\biRic+\frac{4-(n-1)\alpha^2}{4((1-\alpha)n+\alpha)}H^2\geq 0,\]
 we have
\[\int_M |\nabla \varphi|^2+\alpha \Ric^M(x)\varphi^2\geq 0.\]
Then it follows from works of Antonelli--Pozzetta--Xu \cite{antonelli-pozzetta-xu}, Catino--Mari--Mastrolia--Roncoroni \cite{catino-Mari-Mastrolia-Roncoroni} and Hong--Wang \cite{hong-wang-spectral-splitting} (using different ideas) that if $M$ has at least two ends and $\alpha>\frac{n-1}{4}$, then $\Ric^M\geq 0$ on $M.$ By the classical Cheeger--Gromoll splitting theorem, $M$ is a product manifold $\Sigma\times \mathbb{R}$ with linear volume growth, and thus parabolic manifold. 


Thus, we have the following theorem.
 
\begin{theorem}\label{oneend1}
	Let $(X^{n+1},g)$ ($n\leq 5$) be a complete manifold  and 
    $$\alpha\text{-}\biRic_X+\frac{4-(n-1)\alpha^2}{4((1-\alpha)n+\alpha)}H^2\geqslant 0$$
    for $\alpha\in (\frac{n-1}{4},\frac{n}{n-1})$. Let $M^n\hookrightarrow (X^{n+1},g)$ be a complete stable (two-sided if it is minimal) CMC immersion with the mean curvature $H$. Let $K\subset M$ be a compact subset of $M$ with smooth boundary. Then $M\backslash K$ admits at most two unbounded components. In particular, $M$ has at most one non-parabolic end.
\end{theorem} 

In particular, in order to set $\alpha=1$, the dimension $n$ must be smaller than $5$. The following result in minimal case was first shown in \cite[Corollary 1.2]{antonelli-pozzetta-xu}. 

\begin{theorem}\label{oneend}
	Let $(X^{n+1},g)$ ($n\leq 4$) be a complete manifold  with $$\biRic_X+\frac{5-n}{4}H^2\geqslant 0.$$ Let $M^n\hookrightarrow (X^{n+1},g)$ be a complete  stable (two-sided if it is minimal) CMC immersion  with the mean curvature $H$. Let $K\subset M$ be a compact subset of $M$ with smooth boundary. Then $M\backslash K$ admits at most two unbounded components. In particular, $M$ has at most one non-parabolic end.
\end{theorem} 

Therefore, under the assumptions of our main theorems in Section \ref{proof-of-main-results}, it suffices to consider the case where $M^n$ has only one end. We will demonstrate in Section \ref{proof-of-main-results} that this unique end is parabolic.

\subsection{Finite nonparabolic ends of minimal hypersurfaces in general dimensions}
In this subsection, following  \cite{Li-Wang-nonnegatively-curved-manifold} we show that in a manifold $(X^{n+1},g)$ with certain conditions, the number of nonparabolic ends of a stable CMC hypersurface $(M^n,g)$ is at most one.  This fact is known to experts. We recall some details of its proof.

\begin{theorem}\label{oneend3}
Let $(X^{n+1},g)$ be a complete manifold and $M^n\hookrightarrow (X^{n+1},g)$ be a complete stable (two-sided if it is minimal) CMC immersion with the mean curvature $H$. Let $K\subset M$ be a compact subset of $M$ with smooth boundary. Assume that
    \begin{itemize}
        \item[(1)]  either $\Ric^X_{n-1}\geq 0$ and $H=0$,
        \item[(2)] or $\biRic_X+\frac{5-n}{4}H^2\geqslant 0$ and $X$ has bounded geometry.
    \end{itemize}
     Then $M\backslash K$ admits at most one nonparabolic unbounded component. In particular, $M$ has at most one nonparabolic end.
\end{theorem} 
There is no dimension restriction, but only $n\geq 5$ is relevant for our purposes, as the splitting result in lower dimensions has already been established in the previous subsection. Theorem \ref{oneend3} was proved under the assumption $(1)$, so we will focus on assumption $(2)$.

The following inequality of Schoen and Yau is well-known for minimal hypersurface under nonnegative sectional curvature. Note that in dimension 4, Chodosh, Li and Stryker \cite{chodosh-li-stryker} generalized it under weaker assumption $\Ric^X_2\geq 0$. They also mentioned that in higher dimension $n+1$, $\Ric^X_{n-1}\geq 0$ plays an analogous role. To the best of the authors’ knowledge, the most general form of this inequality was established in \cite{Cheng-Cheung-Zhou-CMC}[Lemma 3.2]. It is worth noting that $\Ric^X_{n-1}\geq 0$ implies $\biRic_X\geq0.$  

\begin{lemma}\label{SYL}\cite[Lemma 3.2]{Cheng-Cheung-Zhou-CMC}
	Let $(X^{n+1},g)$ be a complete manifold. Let $M^n\hookrightarrow (X^{n+1},g)$ be a complete stable (two-sided if it is minimal) CMC immersion with the mean curvature $H$. Let $u$ be a harmonic function on $M$. Then
	\begin{equation}\label{SY}
		\int_{M}(\biRic_X(e,\nu)+\frac{5-n}{4}H^2)\phi^2|\nabla u|^2+\frac{1}{n-1}\int_M \phi^2 |\nabla|\nabla u||^2\leqslant \int_M |\nabla \phi|^2|\nabla u|^2
	\end{equation}
for any compactly supported, non-negative function $\phi\in W^{1,2}(M)$ where $e=\nabla u/|\nabla u|.$ 
\end{lemma}

Recall the following result which is first proved for complete orientable minimal hypersurfaces in $\bR^{n+1}$ \cite{Cao-Shen-Zhu-infinitevolume}.

\begin{theorem}[\cite{Frensel-CMC}{Theorem~1}]\label{infinitevol}
    Let $M^n$ be a complete, noncompact manifold, and let $M^n\hookrightarrow X^{n+1}$ be an isometric immersion with bounded mean curvature. If $X^{n+1}$ has bounded geometry, the volume of $M^n$ is infinite.
\end{theorem}

For contradiction, suppose that $\{E_k\}$ and $\{\cE_k\}$ are distinct nonparabolic ends of $M$ adapted to a point $p\in M$ with length scale $L$. We then produce a non-constant harmonic function; see \cite{chodosh-li-stryker}[Lemma 4.3] for detailed proof.
\begin{lemma}\label{harmonic}
    Let $M^n$ be a complete Riemannian manifold. Let $K\subset M$ be a compact subset of $M$ with smooth boundary. Suppose that $M\backslash K$ has at least two nonparabolic unbounded components. Then there exists a non-constant harmonic function with finite Dirichlet energy on $M$.
\end{lemma}

By Lemmas \ref{SYL} and \ref{harmonic}, choosing a standard cutoff function based on the distance function, we can deduce from \eqref{SY} that 
$\biRic_X(e,\nu)+\frac{5-n}{4}H^2\equiv 0$ and $|\nabla u|$ is a nonzero constant. This contradicts the finite energy condition of 
$u$ and Theorem \ref{infinitevol}. Therefore, 
$M$ has at most one nonparabolic end.

\section{CMC hypersurfaces in $5$ or $6$-manifolds}
Without loss of generality, we can assume that $M^n$ is simply-connected by looking at its universal cover. The parabolic ends of $M$ do not affect the construction of cut-off function on $M$ as we will see in the main proof (see also \cite{chodosh-li-stryker}[Remark 5.1]), thus we only need to deal with one nonparabolic end (it possibly does not exist). 

In this section, we adapt the argument from \cite{chodosh-li-stryker} to prove a volume bound that uses the theory of so-called balanced warped $\mu$-bubbles for stable CMC hypersurfaces; see also \cite{Gromov19,chodosh-li-bubble} for relevant statements. 

\subsection{Balanced $\mu$-bubbles}

We recall some preliminary results on existence and stability for balanced warped $\mu$-bubbles. Given a Riemannian $n$-manifold $(M^n,g)$ with boundary, $n\leqslant 7$, we assume that a choice of labeling the components of $\partial M$ is $\partial M=\partial_-M \cup \partial_+M$ where neither of them are empty. For a fixed smooth function $u>0$ on $M$ and a smooth function on $\mathring{M}$ with $h\to \pm\infty$ on $\partial_{\pm}M$, consider the following balanced warped area functional: 
\begin{equation*}
    \mA_{k}(\Omega):=\int_{\partial^{*}\Omega} u^k d\mH^{n-1}-\int_{M}\left(\chi_{\Omega}-\chi_{\Omega_0}\right)hu^k d\mH^n,\quad k>0,
\end{equation*}
for all Caccioppoli sets $\Omega$ with $\Omega\Delta\Omega_0\Subset \mathring{M}$, where the reference set $\Omega_0$ with smooth boundary satisfies 
\begin{equation*}
    \partial \Omega_0\subset \mathring{M}, \quad \partial_+M\subset \Omega_0. 
\end{equation*}
We will call $\Omega$ minimizing $\mA_{k}$ in this class a balanced warped $\mu$-bubble, or just $\mu$-bubble for short.

\begin{remark}
    $\mA_{k}(\Omega)$ is equivalent to the functional
    \begin{equation*}
        \widetilde{\mA}(\Omega)=\int_{\partial^{*}\Omega} d\mH^{n-1}_{\cg}-\int_{M}\left(\chi_{\Omega}-\chi_{\Omega_0}\right)hu^{\frac{-k}{n-1}} d\mH^n_{\cg},
    \end{equation*}
    under the conformal change $\cg=u^{\frac{2k}{n-1}} g$. The critical point of $\widetilde{\mA}$ is a hypersurface with the mean curvature $\tilde{H}=hu^{-\frac{k}{n-1}}$. This is equivalent to the statement in Lemma \ref{first}. We find that it is more convenient to work with $\mA_{k}(\Omega)$ directly instead of $\widetilde{\mA}(\Omega)$ under the conformal change. 
\end{remark}

\subsubsection{Existence of balanced $\mu$-bubbles}
The existence and regularity of a minimizer of $\mA_{k}$ among all Caccioppoli sets (without any equivarient assumptions) was first claimed by Gromov in \cite{Gromov19}, and was rigorously carried out by Zhu in \cite{Zhu-width-estimate}. We omit the detailed proof here and refer the readers to \cite{chodosh-li-bubble}.


\begin{proposition}
There exists a smooth minimizer $\Omega$ for $\mA_{k}$ such that $\Omega\Delta\Omega_0$ is compactly contained in the interior in $M$.
\end{proposition}

\subsubsection{Stability}
We list the first and second variation for a balanced warped $\mu$-bubble without proof; see \cite{chodosh-li-bubble} for detailed calculations.

\begin{lemma}[The first variation]\label{first}
    If $\Omega_t$ is a smooth $1$-parameter family of regions with $\Omega_0=\Omega$ and the normal speed at $t=0$ is $\psi$, then
    \begin{equation*}
        \frac{d}{dt}\mA_{k}(\Omega)=\int_{\Sigma_t}\left(Hu^k+\lp\nabla^{M}u^k,\nu\rp-hu^k\right)\psi d\mH^{n-1}
    \end{equation*}
    where $\nu$ is the outwards pointing unit normal and $H$ is the scalar mean curvature of $\partial \Omega_t$. In particular, a $\mu$-bubble $\Omega$ satisfies
    \begin{equation*}
        H=-ku^{-1}\lp\nabla^{M}u,\nu\rp+h,
    \end{equation*}
    along $\partial\Omega$.
\end{lemma}
Denote $\Sigma=\partial\Omega.$
\begin{lemma}[The second variation]\label{second}
 If $\Omega_t$ is a smooth $1$-parameter family of regions with $\Omega_0=\Omega$ and the normal speed at $t=0$ is $\psi$, then
    \begin{equation*}
        \begin{split}
            \frac{d^2}{d t^2}\bigg|_{t=0}\mA_{k}(\Omega)&=\int_{\Sigma}|\nabla^{\Sigma}\psi|^2 u^k-\left(|A_{\Sigma}|^2+\Ric^M(\nu,\nu)\right)\psi^2 u^k\\
            &+\int_\Sigma k(k-1)u^{k-2}\left(\nabla_{\nu}u\right)^2\psi^2 \\
            &+\int_{\Sigma}ku^{k-1}\left(\Delta_M u-\Delta_{\Sigma} u\right)\psi^2-\left(\nabla_{\nu}(u^kh)\right)\psi^2.
        \end{split}
    \end{equation*}
\end{lemma}
\subsection{Volume bound for bubbles}
In this subsection, we are going to estimate the volume of $\mu$-bubbles constructed in stable two-sided CMC hypersurfaces. We first prove the following two technical propositions.
\begin{proposition}\label{inequalityforAsquare}
    Let $A=(a_{ij})$ be an $n\times n$ symmetric matrix with $n\geq 2$, and let $H=tr(A)$. Then, the following holds:
    \[k|A|^2-a_{11}a_{nn}-\sum_{i=1}^{n}a_{1i}^2-\sum_{i=2}^na_{in}^2+H(a_{11}+a_{nn})\geq \frac{2k^2+k-n+2}{2nk-3n+6}H^2\]
    for any 
    \begin{equation}\label{assumption for k and alpha=1}
    k\geq \max\left\{1,\frac{3(n-2)}{2n}\right\}.
    \end{equation}
    Equality holds if and only if
    \begin{equation*}
        \begin{split}
            a_{11}&=a_{nn}=-\frac{2(n-2k-2)}{2nk-3n+6}H,\\ a_{ii}&=\frac{2k-1}{2nk-3n+6}H, \quad  i\neq 1,n,\\
    a_{ij}&=0, \quad ~~\mbox{otherwise}~~.
        \end{split}
    \end{equation*}
\end{proposition}
\begin{proof}
Combining Cauchy--Schwarz inequality and the fact 
\[|A|^2\geq (a_{11}^2+a_{nn}^2)+\sum_{i=2}^{n-1}a_{ii}^2+2(\sum_{i=2}^na_{1i}^2+\sum_{i=2}^{n-1}a_{in}^2),\]
we have
\begin{align*}
    k|A|^2&-a_{11}a_{nn}-\sum_{i=1}^{n}a_{1i}^2-\sum_{i=2}^na_{in}^2+H(a_{11}+a_{nn})\\
    \geq &(k-1)(a_{11}^2+a_{nn}^2)-a_{11}a_{nn}+\frac{k}{n-2}(H-a_{11}-a_{nn})^2\\
    &+(2k-1)(\sum_{i=2}^na_{1i}^2+\sum_{i=2}^{n-1}a_{in}^2)+H(a_{11}+a_{nn})\\
    \geq &\left(\frac{k}{2}-\frac{3}{4}\right)(a_{11}+a_{nn})^2+\frac{k}{n-2}H^2-\frac{2kH}{n-2}(a_{11}+a_{nn})\\
    &+\frac{k}{n-2}(a_{11}+a_{nn})^2+H(a_{11}+a_{nn})\\
    \geq &\left(\frac{k}{2}-\frac{3}{4}+\frac{k}{n-2}\right)(a_{11}+a_{nn})^2+\left(1-\frac{2k}{n-2}\right)H(a_{11}+a_{nn})+\frac{k}{n-2}H^2.
\end{align*}
The conclusion now follows from elementary calculations. When equality holds, we have $a_{11}=a_{nn}$ and the diagonal entries $a_{ii}$ are all equal for $i=2,\cdots,n-1$.
\end{proof}

Notice that $\frac{3(n-2)}{2n}< 1$ when $n\leq 5$.  Thus, when applying Proposition \ref{inequalityforAsquare}, we can merely assume $k\geq 1$ when $n\leq 5.$ We have a more general result as follows.

\begin{proposition}\label{inequalityforAsquare1}
Let $A=(a_{ij})$ be an $n\times n$ symmetric matrix and denote $H=\tr(A)$. For any $k\geq 1/2$ and $\alpha\leq 2k$ satisfying
\begin{equation}\label{assumptions}
\begin{split}
     &k-\frac{(n-2)\alpha^2}{4k}>0,\\
        \quad k^2-&\frac{1}{4}-\frac{n-2}{4}\left(1+\alpha^2-\frac{\alpha}{k}\right)>0,
\end{split}
\end{equation}
there exists a positive constant $\delta=\delta(n,k,\alpha)$ such that \[I:=k|A|^2-\alpha a_{11}a_{nn}-\alpha\sum_{i=1}^{n-1}a_{1i}^2-\sum_{i=1}^na_{in}^2+H(\alpha a_{11}+a_{nn})\geq \frac{\delta}{n}H^2.\]

\end{proposition}
\begin{proof}
We consider the matrix
\begin{equation}\label{matrix_A} 
P=\begin{pmatrix} k & \frac{1}{2} & \frac{1}{2} & \cdots & \frac{1}{2} \\ \frac{1}{2} & k & \frac{\alpha}{2} & \cdots & \frac{\alpha}{2} \\ \frac{1}{2} & \frac{\alpha}{2} & k &\cdots & 0\\ \vdots & \vdots & \vdots & \ddots & \vdots\\ \frac{1}{2} & \frac{\alpha}{2} & 0 & \cdots & k \end{pmatrix}_{n\times n}.   
\end{equation} 
By elementary row and column operations, it is easy to see that the matrix $P$ is congruent to 
\begin{equation*}
    \begin{pmatrix}
         k-\frac{n-2)}{4k} & \frac{1}{2}\left(1-\frac{(n-2)\alpha}{2k}\right) & 0 & \cdots & 0 \\ \frac{1}{2}\left(1-\frac{(n-2)\alpha}{2k}\right)  & k-\frac{(n-2)\alpha^2}{4k} & 0 & \cdots & 0 \\ 0 & 0 & k &\cdots & 0\\ \vdots & \vdots & \vdots & \ddots & \vdots\\ 0 & 0 & 0 & \cdots & k
    \end{pmatrix}_{n\times n}.
\end{equation*}
Therefore, the positive definiteness of $P$ is equivalent to the positive definiteness of $Q$:
\begin{equation*}
    Q=\begin{pmatrix}
        k-\frac{n-2}{4k} & \frac{1}{2}\left(1-\frac{(n-2)\alpha}{2k}\right)  \\ \frac{1}{2}\left(1-\frac{(n-2)\alpha}{2k}\right)  & k-\frac{(n-2)\alpha^2}{4k}
    \end{pmatrix}
\end{equation*}
which requires 
\begin{equation}\label{two-eigenvalues}
    \begin{split}
       k&-\frac{(n-2)\alpha^2}{4k}>0\\
        \det Q=k^2-\frac{1}{4}-&\frac{n-2}{4}\left(1+\alpha^2-\frac{\alpha}{k}\right)>0.
    \end{split}
\end{equation}

Then we set \[x=\begin{pmatrix}a_{nn}\\ a_{11}\\ a_{22}\\ \vdots\\ a_{n-1,n-1}\\ \end{pmatrix}.\] By assumptions (\ref{assumptions}) and Cauchy--Schwarz inequality, we have \[I\geq x^TPx\geq \delta |x|^2\geq \frac{\delta}{n}H^2\] where $\delta$ is the smallest eigenvalue of $P$.

\end{proof}

We now work on the volume estimate of the warped $\mu$-bubble.
\begin{lemma}\label{diameterforclosedbubble}
Let $(X^{n+1},g)$ $(4\leq n\leq 5)$ be a complete $(n+1)$-dimensional Riemannian manifold, and let $M^n\hookrightarrow (X^{n+1},g)$ be a stable two-sided CMC hypersurface with compact boundary $\partial M$ and the mean curvature $H_M$.
Suppose that there exists $k,\alpha$ satisfying \eqref{assumption for k and alpha=1} or (\ref{assumptions}). Moreover, assume that $k,\alpha$ satisfy
\begin{equation}\label{assumption2}
    \frac{4}{(4-k)\alpha}\leq\frac{n-2}{n-3}, \quad G\geq 0,
\end{equation}
where the matrix $G$ is given in (\ref{cond2}).

Furthermore, if we have
\begin{equation}\label{lambda_k,alpha}
    \lambda_{k,\alpha}+C(k,n)H_M^2\geqslant 2\varepsilon_0>0,
\end{equation}
where $\lambda_{k,\alpha}$ is the infimum of the $(k,\alpha)$-triRicci curvature on $X^{n+1}$ and the constant $C(k,n)$ is given by
\begin{equation}\label{C(k,n)}
C(k,n):=\left\{\begin{array}{cc}
    \frac{2k^2+k-n+2}{2nk-3n+6}, & \alpha=1 \\
  \delta(k,n,\alpha),   & ~~\mbox{otherwise}~~,
\end{array}
\right. 
\end{equation}
then there exist universal constants $L>0$ and $c>0$ such that if $M$ contains a point $p$ satisfying $d_{M}(p,\partial M)>\frac{L}{2}$, then there exists a relatively open set $\Omega\subset B^M_{L/2}(\partial M)$ and a smooth closed $(n-1)$-dimensional Riemannian manifold $\Sigma^{n-1}$ (possibly disconnected) such that $\partial \Omega=\Sigma \cup \partial M$ and the volume of each component, still denoted by $\Sigma$, is bounded by $c$. 

\end{lemma}

\begin{remark}
In this remark, we discuss the admissible choice of $k$ and $\alpha$ that meets several restrictions in the proof of Lemma \ref{diameterforclosedbubble}.
\begin{itemize}
    \item[(1)] When $n=4$, we may choose $k=1$, $\alpha=1$. It satisfies the assumption \eqref{assumption for k and alpha=1}.
For the matrix $G$ defined in (\ref{cond2}), we can take $\beta=1/3$ such that
\begin{equation*}
    \frac{1}{k}-\beta=\frac{2}{3}>0,\quad 1-\alpha \frac{n-2}{n-1}=\frac{1}{3}>0, \quad \det G=\frac{1}{36}>0,
\end{equation*}
i.e., $G$ is positive definite. Note that $\frac{4}{(4-k)\alpha}\leqslant 2=\frac{n-2}{n-3}$ which verifies \eqref{assumption2}. By \cite[Theorem 1]{antonelli-xu}, there exists a universal constant $c>0$ such that
\begin{equation*}
    \Vol(\Sigma)\leqslant c.
\end{equation*}
On the other hand, condition \eqref{lambda_k,alpha} becomes
\[C_3+\frac{1}{2}H_M^2\geq 2\varepsilon_0\] where $C_3$ is the $3$-intermediate curvature.

\item[(2)] When $n=5$, the choice of $k$ and $\alpha$ becomes very delicate. First, it is not hard to check that when $\alpha=1$, there does not exist $k\geq 1$ satisfying \eqref{assumption2} and \eqref{assumption for k and alpha=1}. Moreover, when $\alpha$ is slightly bigger than $1$, condition \eqref{assumptions}: 
\begin{equation*}
     \quad k^2+\frac{3\alpha}{4k}>\frac{3}{4}\alpha^2+1>\frac{3}{4}\alpha+1
\end{equation*}
requires $k>1$ but the condition $\det G>0$:
\begin{equation*}
    \frac{1}{k}>\frac{\beta\left((1-\frac{3\alpha}{4})(\frac{3\alpha}{16}-\frac{3}{4})-\frac{3\alpha^2}{64}\right)+\frac{1}{4}(1-\frac{3\alpha}{4})}{-\frac{3}{16}\alpha^2+\left(\frac{3}{16}-\frac{3}{4}(1-\beta)\right)\alpha-\frac{3}{16}+1-\beta}
\end{equation*}
requires $k<1$. In all, we may choose $\alpha=0.99$, $k=1.001$, and $\beta=0.001$. Direct calculations yield
\begin{equation*}
\begin{split}
    k-\frac{3}{4k}\alpha^2&=0.2667,\\
    k^2-\frac{1}{4}-\frac{3}{4}\left(1+\alpha^2-\frac{\alpha}{k}\right)&=0.0087,
    \end{split}
\end{equation*}
which verifies assumptions (\ref{assumptions}). Besides, we have
\begin{equation*}
    \frac{1}{k}-\beta=0.998>0, \quad 1-\alpha\frac{3}{4}=0.2575>0, \quad \det G=0.0017.
\end{equation*}
So $G$ is positive definite.  Moreover, $\frac{4}{(4-k)\alpha}=1.3472\leq\frac{3}{2}=\frac{n-2}{n-3}$ satisfies \eqref{assumption2}, we have by \cite[Theorem 1]{antonelli-xu} that
\begin{equation*}
    \Vol(\Sigma)\leqslant c.
\end{equation*}

\end{itemize}
\end{remark}

\begin{proof}
Let $\varphi_0$ be a smoothing of the function $d(\partial M, \cdot)$ so that $|\nabla \varphi_0|\leqslant 2$ and $\varphi_0\big|_{\partial M}\equiv 0$. Let $\epsilon\in (0,\frac{1}{2})$ so that $\epsilon$ and $\frac{2}{\sqrt{\beta\varepsilon_0}}\pi+2\epsilon$ are regular values of $\varphi_0$, where the positive constant $\beta$ will be determined later.

Define
\begin{equation*}
    \varphi=\frac{\varphi_0-\epsilon}{\frac{2}{\sqrt{\beta\varepsilon_0}}+\frac{\epsilon}{\pi}}-\frac{\pi}{2}.
\end{equation*}
$\Omega_1=\{x\in M:-\frac{\pi}{2}<\varphi<\frac{\pi}{2}\}$, and $\Omega_0=\{x\in \Sigma^{'}:-\frac{\pi}{2}<\varphi \leqslant 0\}$. Clearly $|\Lip(\varphi)|\leqslant \sqrt{\beta\varepsilon_0}$, and we define
\begin{equation*}
    h=-\sqrt{\frac{\varepsilon_0}{\beta}}\tan (\varphi).
\end{equation*}
We compute that
\begin{equation*}
    \nabla h=-\sqrt{\frac{\varepsilon_0}{\beta}}\left(1+\tan^2(\varphi)\right)\nabla \varphi=-\sqrt{\frac{\varepsilon_0}{\beta}}\left(1+\frac{\beta h^2}{\varepsilon_0}\right)\nabla \varphi,
\end{equation*}
so we have
\begin{equation*}
    |\nabla h|\leqslant \varepsilon_0\left(1+\frac{\beta h^2}{\varepsilon_0}\right).
\end{equation*}

We now consider the functional $\mA_{k}$ on $M^{n}$, where the weight function $u$ satisfies
 \begin{equation*}
 \begin{split}
     -\Delta_{M}u&=\left(|A_{M}|^2+\Ric^X(\eta, \eta)\right)u ~~\mbox{in}~~ \mathring{M},\\
     u&=1 ~~\mbox{on}~~ \partial M;
     \end{split}
 \end{equation*}
 see \cite{Wu23,chen-hong-li-docarmoquestion} for more details on the existence of $u$.
We assume that $\Omega$ is the minimizer of $\mA_{k}$ and $\Sigma$ is a connected component of $\partial \Omega\backslash \partial M$. Let $\nu$ be the unit normal vector to $\Sigma$.

By Lemma \ref{second}, we have
\begin{equation*}
 \begin{split}
     \int_{\Sigma}u^k |\nabla^{\Sigma}\psi|^2-ku^{k-1}\psi^2\Delta_{\Sigma}u\geqslant& \int_{\Sigma} \left(|A^{\Sigma}|^2+\Ric^M(\nu, \nu)-ku^{-1}\Delta_M u\right)u^k\psi^2\\
     &+\int_{\Sigma} \nabla_\nu(u^k h)\psi^2-k(k-1)u^{k-2}(\nabla_\nu u)^2\psi^2.
     \end{split}
 \end{equation*}
Plugging $\phi=u^{\frac{k}{2}}\psi$ into the above inequality yields
 \begin{align*}
     \int_\Sigma |\nabla\phi|^2+\frac{k^2}{4}u^{-2}&|\nabla u|^2\phi^2-ku^{-1}\phi\nabla\phi\cdot\nabla u-k\phi^2u^{-1}\Delta_\Sigma u\\ 
     \geq & \int_{\Sigma} \left(|A^{\Sigma}|^2+\Ric^M(\nu, \nu)-ku^{-1}\Delta_M u\right)\phi^2\\
     &+\int_{\Sigma} kh(\nabla_\nu \log u) \phi^2+\nabla_\nu h \phi^2-k(k-1)(\nabla_\nu \log u)^2\phi^2.
 \end{align*}
 Integration by parts yields
 \begin{align*}
     \int_\Sigma |\nabla\phi|^2&+\left(\frac{k^2}{4}-k\right)u^{-2}|\nabla u|^2\phi^2+ku^{-1}\phi\nabla\phi\cdot\nabla u\\ \geq & \int_{\Sigma} \left(|A^{\Sigma}|^2+\Ric^M(\nu, \nu)-ku^{-1}\Delta_M u\right)\phi^2\\
     &+\int_{\Sigma} kh(\nabla_\nu \log u) \phi^2+\nabla_\nu h \phi^2-k(k-1)(\nabla_\nu \log u)^2\phi^2.
 \end{align*}
 By Cauchy-Schwarz inequality, we obtain
  \begin{align*}
    \frac{4}{4-k} \int_\Sigma |\nabla\phi|^2&\geq  \int_{\Sigma} \left(|A^{\Sigma}|^2+\Ric^M(\nu, \nu)-ku^{-1}\Delta_M u\right)\phi^2\\
     &+\int_{\Sigma} kh(\nabla_\nu \log u) \phi^2+\nabla_\nu h \phi^2-k(k-1)(\nabla_\nu \log u)^2\phi^2.
 \end{align*}
Substituting $h$ by $H^{\Sigma}+k\nabla_\nu \log u$ yields
   \begin{align}\label{equation 11}
    \frac{4}{4-k} \int_\Sigma |\nabla\phi|^2\geq & \int_{\Sigma} \left(|A^{\Sigma}|^2+\Ric^M(\nu, \nu)-ku^{-1}\Delta_M u\right)\phi^2\\\nonumber
     &+\int_{\Sigma} k H^{\Sigma}(\nabla_\nu \log u) \phi^2+\nabla_\nu h \phi^2+k(\nabla_\nu \log u)^2\phi^2.
 \end{align}
 
 We next deal with the curvature term
 \begin{equation}\label{curvatureterm}
     -ku^{-1}\Delta_{M}u+|A^{\Sigma}|^2+\Ric^M(\nu, \nu).
 \end{equation}
We select a local orthonormal frame $\{e_i\}_{i=1}^{n-1}$ on $\Sigma^{n-1}$ and extend it by assigning $\nu=e_n$ and $\eta=e_{n+1}$ so that $\{e_1,\cdots,e_{n-1},\nu,\eta\}$ is a local orthonormal frame of $X^{n+1}$. By Gauss equation, we have
 \begin{equation*}
 \begin{split}
     \Ric^{\Sigma}(e_1,e_1)&=\sum_{i=1}^{n-1}\R^{\Sigma}_{i1i1}=\sum_{i=1}^{n-1}\left(\R^M_{i1i1}+A^{\Sigma}_{ii}A^{\Sigma}_{11}-\left(A^{\Sigma}_{i1}\right)^2\right)\\
     &=\sum_{i=1}^{n-1}\R^M_{i1i1}+H^{\Sigma}A^{\Sigma}_{11}-\sum_{i=1}^{n-1}\left(A^{\Sigma}_{i1}\right)^2\\
     &=\sum_{i=1}^{n-1}\R^X_{i1i1}+(H^M-A^M_{nn})A^M_{11}-\sum_{i=1}^{n-1}(A^M_{i1})^2+H^{\Sigma}A^{\Sigma}_{11}-\sum_{i=1}^{n-1}\left(A^{\Sigma}_{i1}\right)^2.
     \end{split}
 \end{equation*}
 On the other hand,
 \begin{align*}
     \Ric^M(\nu,\nu)&=\sum_{i=1}^n\left(R^X_{inin}+A^M_{nn}A_{ii}^M-A_{ni}^2\right)\\
     &=\sum_{i=1}^n R^X_{inin}+H^M A^M_{nn}-\sum_{i=1}^n (A_{in}^M)^2.
 \end{align*}
 Therefore, 
 \begin{equation*}
 \begin{split}
     -ku^{-1}&\Delta_{M}u+|A^{\Sigma}|^2+\Ric^M(\nu, \nu)+\alpha\Ric^\Sigma(e_1,e_1)\\
     =&k|A^M|^2+k\Ric^X(\eta,\eta)+\Ric^M(\nu,\nu)+|A^\Sigma|^2+\alpha\Ric^\Sigma(e_1,e_1)\\
     =&k\sum_{i=1}^{n+1}\R^X_{i(n+1)i(n+1)}+\sum_{i=1}^{n}\R^X_{inin}+\alpha\sum_{i=1}^{n-1}\R^X_{i1i1}\\
     &+k|A^M|^2+\alpha(H^{M}-A^M_{nn})A^M_{11}-\alpha\sum_{i=1}^{n-1}\left(A^{M}_{i1}\right)^2+H_MA^M_{nn}-\sum_{i=1}^{n}\left(A^{M}_{in}\right)^2\\
     &+|A^\Sigma|^2+ \alpha H^\Sigma A^\Sigma_{11}-\alpha\sum_{i=1}^{n-1}
(A^\Sigma_{i1})^2\\    
\geq& k\sum_{i=1}^{n+1}\R^X_{i(n+1)i(n+1)}+\sum_{i=1}^{n}\R^X_{inin}+\alpha\sum_{i=1}^{n-1}\R^X_{i1i1}\\
+& C(k,n)H_M^2
+|A^\Sigma|^2+ \alpha H^\Sigma A^\Sigma_{11}-\alpha\sum_{i=1}^{n-1}
(A^\Sigma_{i1})^2\
     \end{split}
 \end{equation*}
where the last inequality is due to
Proposition \ref{inequalityforAsquare} or Proposition \ref{inequalityforAsquare1} and $C(k,n)$ is defined in (\ref{C(k,n)}).

Now the inequality \eqref{equation 11} becomes
  \begin{align*}
    \frac{4}{4-k}& \int_\Sigma |\nabla\phi|^2\geq   \int_{\Sigma} (-\alpha\Ric^\Sigma(e_1,e_1))\phi^2\\
    &+\int_{\Sigma}\left(k\sum_{i=1}^{n+1}\R^X_{i(n+1)i(n+1)}+\sum_{i=1}^{n}\R^X_{inin}+\alpha\sum_{i=1}^{n-1}\R^X_{i1i1}+C(k,n)H_M^2 \right)\phi^2\\\nonumber
   &+ \int_\Sigma \left(|A_\Sigma|^2+ \alpha H^\Sigma A^\Sigma_{11}-\alpha\sum_{i=1}^{n-1}
(A^\Sigma_{i1})^2\right)\phi^2\\
     &+\int_{\Sigma} k H^\Sigma(\nabla_\nu \log u) \phi^2+\nabla_\nu h \phi^2+k(\nabla_\nu \log u)^2\phi^2.
 \end{align*}

We now estimate the following quantity 
 \[I:=|A^\Sigma|^2+ \alpha H^\Sigma A^\Sigma_{11}-\alpha\sum_{i=1}^{n-1}
(A^\Sigma_{i1})^2+kH^\Sigma(\nabla_\nu \log u)+\nabla_\nu h+ k(\nabla_\nu \log u)^2\]
and expect that $I\geq \beta h^2$ for some $\beta>0$.
In the same spirit in \cite[Page 11-12]{mazet}, it is equivalent to show that the matrix
\begin{equation}\label{cond2}
G=\begin{pmatrix}
    \frac{1}{n-1}+\frac{\alpha}{n-1}-\frac{\alpha}{(n-1)^2}+\frac{1}{k}-1 & \frac{\alpha}{2}\sqrt{\frac{n-2}{n-1}}(1-\frac{2}{n-1}) & \frac{1}{2}-\frac{1}{k} \\
    \frac{\alpha}{2}\sqrt{\frac{n-2}{n-1}}(1-\frac{2}{n-1}) & 1-\alpha\frac{n-2}{n-1} & 0\\
    \frac{1}{2}-\frac{1}{k} & 0 & \frac{1}{k}-\beta 
\end{pmatrix} 
\end{equation}
is positive semi-definite for some $k$, $\alpha$ and $\beta$. If so, by our assumptions, we obtain that
 \begin{align*}
    \frac{4}{4-k} \int_\Sigma |\nabla\phi|^2&\geq  \int_{\Sigma} (-\alpha\Ric^\Sigma(e_1,e_1)+2\varepsilon_0+\beta h^2+\nabla_\nu h)\phi^2.
 \end{align*}

According to the definition of $h$ and choosing $e_1$ realizing the minimum Ricci curvature $\lambda_{\Ric}(\Sigma)=\Ric_{\Sigma}(e_1,e_1)$, we get
\begin{equation*}
\begin{split}
    \frac{4}{4-k}\int_{\Sigma}|\nabla^{\Sigma}\phi|^2  \geqslant \int_{\Sigma}\left(- \alpha\lambda_{\Ric}(\Sigma)+\varepsilon_0\right)\phi^2.
    \end{split}
\end{equation*}
Therefore, there is a positive smooth function $\theta\in C^{\infty}(\Sigma)$ so that
\begin{equation*}
    \gamma\Delta_\Sigma \theta\leq \lambda_{\Ric} \theta-\frac{\varepsilon_0}{\alpha}\theta, \quad \gamma=\frac{4}{(4-k)\alpha}.
\end{equation*}
We need 
\[\gamma\leq \frac{n-2}{n-3}\]
to hold so that we can apply \cite[Theorem 1]{antonelli-xu} to conclude that there exists a constant $c$ such that $\Vol(\Sigma)\leq c.$

\end{proof}

\subsection{Almost linear volume growth} 
In this part, we will upgrade the volume estimate of $\mu$-bubbles to the volume growth of ends in stable CMC hypersurfaces under several curvature conditions.

\begin{theorem}\label{volumeforslice}
    Let $(X^{n+1},g)$ $(4\leq n\leq 5)$ be a complete $(n+1)$-dimensional Riemannian manifold, and let $M^n\hookrightarrow (X^{n+1},g)$ be a simply-connected stable two-sided CMC hypersurface. Let $\{E_l\}_{l\in\bN}$ be an end of $M^n$ adapted to $p\in M$ with length scale $L$, where $L$ is the constant from Lemma \ref{diameterforclosedbubble}. Denote $M_l:=E_l\cap B_{(l+1)L}(p)$. Assume that either $X$ has bounded geometry or $X$ has weakly bounded geometry and $\Ric^X_{n-1}\geq C_2$ for a constant $C_2$. Under assumptions in Lemma \ref{diameterforclosedbubble}, there is a universal constant $C>0$ independent of $l$ such that
    \begin{equation*}
        \Vol(M_l)<C
    \end{equation*}
    for all $l$ sufficiently large.
\end{theorem}

 
\begin{proof}
We apply Lemma \ref{diameterforclosedbubble} to the end $E_l$, which supplies an open set $\Omega_l$ in $B_{L/2}(\partial E_l)\cap E_l$ with $\partial E_l\subset \partial \Omega_l$. Note that $\Omega_l\subset M_l$ and $\partial E_{l+1}\cap \overline{\Omega_l}=\emptyset$ by construction. 

As claimed in \cite[Lemma 5.4]{chodosh-li-stryker}, $\overline{M_l}$ is connected and there exists a component $\Sigma^{*}_l$ of the $\mu$-bubble $\Sigma_l:=\partial \Omega_l\backslash \partial E_l$ separating $\partial E_l$ and $\partial E_{l+1}$. In fact, if not, we can find a non-contractible closed curve, contradicting the simply-connectedness of $M^n$. By Lemma \ref{diameterforclosedbubble}, we have $\Vol(\Sigma_l^{*})\leqslant c$ where $c$ is independent of $l.$

We now claim that there is a constant $D>0$ independent of $l$ such that $\Vol(M_l)<D$. It can be deduced by using the area comparison theorem of hypersurfaces.  This idea was firstly applied by Wei-Xu-Zhang in \cite{Wei_Xu_Zhang}. Due to the fact that $\Sigma_{l-1}^{*}$ separates $p$ and $E_l$, every hypersurface $\partial B_s(p)\cap M_l$ for $s\in [lL,(l+1)L]$ lies behind $\Sigma_{l-1}^{*}$, which means any geodesic starting from $p$ to $\partial B_s(p)\cap M_l$ goes through $\Sigma_{l-1}^{*}$. 

It follows from Lemma \ref{estimate} and Lemma \ref{Ric-lower} that $\Ric^M\geq -Cg$. Then by Bishop--Gromov volume comparison theorem, we have
    \[\Vol(\partial B_s(p)\cap M_l)\leq C\frac{e^{(n-1)s}}{e^{(n-1)(l-1)L}}\Vol(\Sigma_{l-1}^{*})\leq Ce^{2(n-1)L}\]
    for any $s\in [lL,(l+1)L]$. Integrating and using co-area formula yield
    \[\Vol(M_l)\leq CLe^{2(n-1)L}.\]
\end{proof}

\section{Proof of main results}\label{proof-of-main-results}
We are going to prove a more general result than ones stated in the introduction. For clarity, we state our result and divide proofs into different cases.

\subsection{Case 1: $n=4$}



\begin{theorem} \label{maintheorem5.1}   Let $(X^{5},g)$ be a complete manifold and let $M^4\hookrightarrow (X^{5},g)$ be a complete noncompact stable CMC immersion with  mean curvature $H$. 

    $(1)$ If $\Ric^X+\frac{1}{4}H^2\geq 0$, $\biRic_X+\frac{1}{4}H^2\geq 0$, $X$ has bounded geometry, and $\lambda_k+\frac{2k^2+k-2}{8k-6}H^2\geq 2\varepsilon_0$ for a positive constant $\varepsilon_0$ and some $k\in[1,2]$.  Then $M$ is parabolic with at most two ends. In particular, $M$ is umbillic and $\Ric^X(\nu,\nu)+\frac{1}{4}H^2=0$ along $M$.

    $(2)$ If $\Ric^X_{3}+\frac{3}{28}H^2 \geq 0$, $X$ has weakly bounded geometry, and $\lambda_k+\frac{2k^2+k-2}{8k-6}H^2\geq 2\varepsilon_0$ for a positive constant $\varepsilon_0$ and some $k\in[1,2]$.  Then $M$ is parabolic with at most two ends. In particular, $M$ is umbillic and $\Ric^X(\nu,\nu)+\frac{1}{4}H^2=0$ along $M$.
\end{theorem}

\begin{remark}\label{curvaturerelations}
    Note that $\Ric^X_{3}+\frac{3}{28}H^2 \geq 0$ implies $\Ric^X+\frac{1}{4}H^2\geq 0$ and $\biRic_X+\frac{1}{4}H^2\geq 0$.
\end{remark}
\begin{proof}
First, observe that the universal cover of 
$M$ is stable, so without loss of generality, we may assume that 
$M$ is simply connected. The theorem then follows by constructing suitable compactly supported test functions. To achieve this, we decompose 
$M$ into several geometric building blocks, following the approach in \cite{chodosh-li-stryker}. Let 
$M^n$ be a complete, simply connected Riemannian 
$n$-manifold.

\textbf{\textbf{(1):}}  It follows from Theorem \ref{oneend1} that $M$ has at most two ends when $\biRic_X+\frac{1}{4}H^2\geq 0$. Furthermore, if there exist more than two ends, $M$ is a product manifold  with two ends. In the sequel, we only need to address the one end case. We denote $\{E_l\}$ the unique end of $M$ adapted to $p\in M$ with length scale $L$.

We now demonstrate the decomposition of $M$. For fixed $l_0\geqslant 1$, we set
\begin{itemize}
    \item $M_l:=E_l\cap B_{(l+1)L}(p)$ for all $l$.
    \item $\{P^{(\alpha)}_{l_0}\}_{\alpha=1}^{n_{l_0}}$ are the components of $M\backslash \oB_{l_0L}(p)$ besides $E_{l_0}$.
    \item $\{P^{(\alpha)}_{l}\}_{\alpha=1}^{n_{l}}$ are the components of $E_{l-1}\backslash \oB_{lL}(p)$ besides $E_{l}$ for $l>l_0$.
    
\end{itemize}
By the uniqueness of the end, all components $P_l^{(\alpha)}$ are bounded. It also should be noted that the number $n_l$ of components $P_l^{(\alpha)}$ may be unbounded as $l$ tends to infinity. For fixed $i \geq 1$, we have the decomposition (see Figure 1)
\begin{equation*}
    M=\oB_{l_0L}(p)\cup\bigcup_{l=l_0}^{l_0+i-1}\left(\oM_l\cup\bigcup_{\alpha=1}^{n_l}\oP_l^{(\alpha)}\right)\cup \left(\oE_{l_0+i-1}\backslash B_{(l_0+i)L}(p)\right).
\end{equation*}

We now construct the desired cutoff function. For each $l$, let $\rho_l$ be a smooth function on $\oM_l$ such that $|\nabla \rho_l|\leqslant 2$,
    \begin{equation*}
        \rho_l\big|_{\partial E_l}\equiv lL, \quad \rho_l\big|_{\partial M_l\backslash \partial E_l}\equiv (l+1)L.
    \end{equation*}
    For instance, we can take $\rho_l$ to be a smoothing of the distance function from $p$. Let $\phi:[l_0L,(l_0+i)L]\to [0,1]$ be the linear function defined by
    \[\phi(t)=1+\frac{1}{i}(l_0-\frac{t}{L}).\] Define a compactly supported Lipschitz function on $M$ as follows:
    \[\varphi(x)=\begin{cases}
        1 \ \ \ \ \ \ \ \ \ \ & x \in \overline{B}_{l_0L}(p),\\
        \phi(\rho_l(x)) & x\in  \overline{M_l}, \ \ \l_0\leq l<l_0+i,\\
        \phi(lL)  & x\in \oP_l^{(\alpha)}, \ \ l_0\leq l<l_0+i,\\
        0 &  \text{otherwise}.
    \end{cases}
    \]

For any positive superharmonic function $\omega$ on $M$, we define $\tau=\log \omega$ which satisfies $\Delta \tau\leq -|\nabla\tau|^2.$ Integration by parts and Cauchy-Schwarz inequality yield
\begin{align*}
    \int_M|\nabla\tau|^2\varphi^2&\leq -\int_M \varphi^2\Delta \tau\\
    &=\int_M2\varphi\nabla\varphi\cdot\nabla\tau\\
    &\leq \delta\int_M|\nabla\varphi|^2+\frac{1}{\delta}\int_M|\nabla\tau|^2\varphi^2.
\end{align*}
Choosing $\delta>1$, we have
\begin{align*}
    \int_M|\nabla\tau|^2\varphi^2&\leq C\int_M|\nabla\varphi|^2\\
    &\leq C\sum_{l=l_0}^{l_0+i-1}\int_{\bar{M}_l}\frac{1}{(iL)^2}|\nabla \rho_l|^2\\
    &\leq C\frac{1}{iL^2}\to 0, \quad i \to 0.
\end{align*}
Here we have used the volume estimate in Theorem \ref{volumeforslice} in the last inequality. Thus, both $\tau$ and $\omega$ are constant functions, and $M$ is parabolic.

Then, by the stability of $M$, there is a positive superharmonic function $\omega$ on $M$ satisfying
\[\Delta \omega=-|A|^2\omega-\Ric^X(\nu,\nu)\omega\leq -\frac{1}{4}H^2\omega-\Ric^X(\nu,\nu)\omega\leq 0.\]

By previous discussion, we conclude that $M$ is umbillic and has $\Ric^X(\nu,\nu)+\frac{1}{4}H^2=0$ along $M.$

\begin{figure}[h]

\includegraphics[width=13cm]{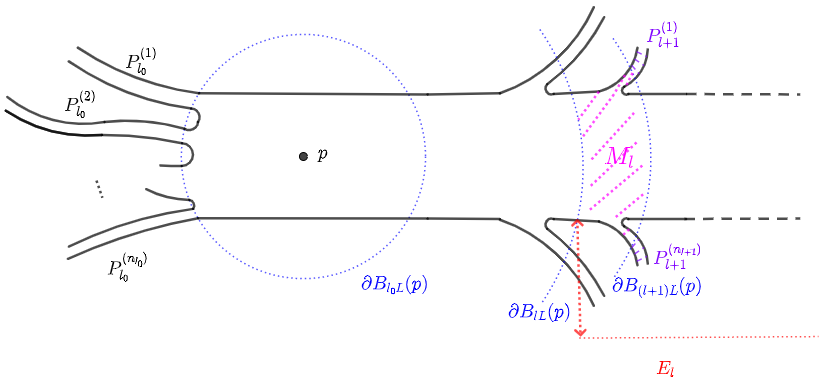}
\caption{Decomposition of $M$}
\end{figure}
\vskip.2cm

\textbf{\textbf{(2):}}
By Remark \ref{curvaturerelations} and Theorem \ref{oneend1}, $ M$ has at most two ends. Then the proof is exactly same as the one in \textbf{\textbf{(1)}}. Notice that we can apply Lemma \ref{diameterforclosedbubble} and Theorem \ref{volumeforslice} under the assumptions of \textbf{\textbf{(2)}}.

\end{proof}

\subsection{\textbf{Case 2:} $n=5$}\label{n=5}

\begin{theorem}\label{maintheorem5.3}    Let $(X^{6},g)$ be a complete manifold and let $M^5\hookrightarrow (X^{6},g)$ be a complete noncompact stable CMC immersion with  mean curvature $H$. 

    $(1)$ If $\Ric^X+\frac{1}{5}H^2\geq 0$, $\biRic_X\geq 0$, $X$ has bounded geometry, and $\lambda_{(0.99,1.001)}+\delta(0.99,5,1.001)H^2\geq 2\varepsilon_0$ for a positive constant $\varepsilon_0$ where $\delta$ is taken from Proposition \ref{inequalityforAsquare1}.  Then $M$ is parabolic. In particular, $M$ is umbillic and $\Ric^X(\nu,\nu)+\frac{1}{5}H^2=0$ along $M$.

    $(2)$ If $\Ric^X_{4}\geq 0$, $H=0$, $X$ has weakly bounded geometry, and $\lambda_{(0.99,1.001)}\geq 2\varepsilon_0$ for a positive constant $\varepsilon_0$.  Then $M$ is parabolic. In particular, $M$ is totally geodesic and $\Ric^X(\nu,\nu)=0$ along $M$.
\end{theorem}

\begin{proof}

\textbf{(1) and (2):} Under current assumptions, there is no priori upper bound on the number of ends. Fortunately, by Theorem \ref{oneend3}, we know that the number of nonparabolic end is at most one if either $\biRic_X\geq 0$ and $X$ has bounded geometry or $\Ric^X_{4}\geq 0$ and $M$ is minimal.

Then the rest of proof is basically same as that of Theorem \ref{maintheorem5.1} \textbf{(1)}, except
 that the component $\oP_k^{(\alpha)}$ now is not necessarily bounded. On each unbounded parabolic component $\oP_k^{(\alpha)}$, the modified test function is given by the following well-known fact:  
for any $l$ and $\alpha\in [1,n_l]$ and for any small $\epsilon>0$, there exists a positive compactly supported Lipschitz function $u_{l,\alpha}$ on $P_l^{(\alpha)}$ such that
\begin{equation}\label{functionattachtoparabolicends}
    u_{l,\alpha}|_{\partial P_l^{(\alpha)}}=1 \ \text{and}\ \ \int_{P_l^{(\alpha)}}|\nabla u_{l,\alpha}|^2<\epsilon.
\end{equation}
This fact follows from the definition of parabolic ends (see \cite{peterwangspectrum} or \cite[Proposition 3.4]{chodosh-li-stryker} for more details). Thus on $\oP_l^{(\alpha)}$, we now define $\varphi:=\phi(lL)u_{l,\alpha}$. Besides, on bounded components, we set $\varphi:=\phi(lL)$ as we previously did. 

We can choose $\epsilon$ sufficiently small in (\ref{functionattachtoparabolicends}) such that the energy of all parabolic ends on a compact subset in $M$ is sufficiently small. Moreover, under current assumptions, we can apply Lemma \ref{diameterforclosedbubble} and Theorem \ref{volumeforslice}.

\end{proof}




\appendix
\section{Example}
We provide a (in fact, a family of) closed $5$-dimensional Riemannian manifold $(X^5,g)$ with bounded geometry, satisfying
    \[\operatorname{Ric}_X\geq 0,\ \, \biRic_X\geq 0,\ \, \text{and}\ \ 1\operatorname{-triRic}_X\geq \delta>0, \] while $\Ric_3<0$ somewhere. 

Consider the Berger sphere: $\bS^5=\oG/\oH$, where $\oG=SU(3)\times \bS^1$ and $\oH=SU(2)\times \bS^1$. We follow the notation in \cite{Ziller-Bergersphere}. For the bi-invariant metric $\lp X,Y\rp=-\frac{1}{2}\tr (XY)$, $\kg=\kh\oplus\kp$ is an $\Ad(H)$ invariant splitting and $\kp=\kh^{\perp}$ has an orthonormal basis: for $r\in\{1,2\}$,
\begin{equation*}
    A=\frac{i}{\sqrt{3}}\left(E_{1,1}+E_{2,2}-2E_{3,3}\right),\quad e_r=E_{r,3}-E_{3,r}, \quad f_r=i\left(E_{r,3}+E_{3,r}\right).
\end{equation*}
Moreover, $\kp=\kp_1\oplus \kp_2$, where
\begin{equation*}
    \kp_1:=\vspan \{e_r,f_r\}, \quad \kp_2=\bR\cdot A.
\end{equation*}
Let $D$ be a basis of the Lie algebra of $\bS^1$. By \cite{Ziller-Bergersphere}[Theorem 3], the metric $g_s$ defined by
\begin{equation*}
    g_s:=g|_{\kp_1\times\kp_1}+s^2\cdot g|_{\kp_2\times\kp_2}
\end{equation*}
is naturally reductive with respect to the decomposition $\overline{\kg}=\kh\oplus \overline{\kp}_s$, where
\begin{equation*}
    \overline{\kp}_s:=\vspan \{s^2A+(s^2-1)D, e_r,f_r\}.
\end{equation*}
We will abbreviate
\begin{equation*}
    d_s=\frac{s^2A+(s^2-1)D}{s}.
\end{equation*}
Thus $\{d_s, e_r, f_r\}$, $r\in\{1,2\}$, is a $g_s$-orthonormal basis of $\overline{\kp}_s$. Since $\Ad(\oH)$ maps any unit vector into a vector
\begin{equation*}
    v_{\alpha}:=\cos \alpha d_s+\sin \alpha e_1,
\end{equation*}
without loss of generality, we can restrict ourselves to $v=v_{\alpha}$. For convenience, we set $\ove_1=\sin \alpha d_s-\cos \alpha e_1$.

Following the calculation in \cite{Ziller-Bergersphere}, we list sectional curvatures of $(\bS^5, g_s)$:
\begin{equation*}
\begin{split}
    \sec_{g_s}(v, \ove_1)=\frac{3}{4}s^2, \quad & \sec_{g_s}(v, e_2)=\sin^2 \alpha+\frac{3s^2}{4}\cos^2\alpha,\\
    \sec_{g_s}(v, f_1)=\frac{3}{4}s^2+(4-3s^2)\sin^2 \alpha, \quad& \sec_{g_s}(v, f_2)=\sin^2 \alpha+\frac{3s^2}{4}\cos^2\alpha,\\
    \sec_{g_s}(\ove_1,e_2)=\cos^2 \alpha+\frac{3s^2}{4}\sin^2\alpha, \quad & \sec_{g_s}(\ove_1,f_1)=\frac{3}{4}s^2+(4-3s^2)\cos^2 \alpha,\\
    \sec_{g_s}(\ove_1,f_2)=\cos^2 \alpha+\frac{3s^2}{4}\sin^2\alpha, \quad & \sec_{g_s}(e_2,f_1)=1,\\
    \sec_{g_s}(e_2,f_2)=4-\frac{9}{4}s^2, \quad & \sec_{g_s}(f_1,f_2)=1.
    \end{split}
\end{equation*}

Elementary calculation yields
\begin{equation*}
    \begin{split}
        \Ric_{g_s}(v,v)&=3s^2+\left(6-\frac{9}{2}s^2\right)\sin^2 \alpha, \\
        \Ric_{g_s}(\ove_1,\ove_1)&=6-\frac{3}{2}s^2+\left(\frac{9}{2}s^2-6\right)\sin^2 \alpha\\
        \Ric_{g_s}(e_2,e_2)=\Ric_{g_s}(f_1,f_1)&=\Ric_{g_s}(f_2,f_2)=6-\frac{3}{2}s^2,\\
        \Ric_{g_s}& \equiv 0,  ~~\mbox{otherwise.}~~
    \end{split}
\end{equation*}
When $s^2\in (16/9,4]$, we have
\begin{equation*}
\inf \sec_{g_s}=4-\frac{9}{4}s^2<0, \quad     \inf \Ric_{g_s}=6-\frac{3}{2}s^2\geq 0.
\end{equation*}

Next, we list all the bi-Ricci curvatures:
\begin{equation*}
    \begin{split}
        \biRic_{g_s}(v,\ove_1)&=6+\frac{3}{4}s^2,\quad \biRic_{g_s}(v,f_1)=6+\frac{3}{4}s^2+\left(2-\frac{3}{2}s^2\right)\sin^2 \alpha,\\
        \biRic_{g_s}(v,e_2)&=\biRic_{g_s}(v,f_2)=6+\frac{3}{4}s^2+\left(5-\frac{15}{4}s^2\right)\sin^2 \alpha,\\
        \biRic_{g_s}(\ove_1,e_2)&=\biRic_{g_s}(\ove_1,f_2)=11-3s^2+\left(-5+\frac{15}{4}s^2\right)\sin^2 \alpha,\\
        \biRic_{g_s}(\ove_1,f_1)&=8-\frac{3}{4}s^2+\left(-2+\frac{3}{2}s^2\right)\sin^2 \alpha,\quad \biRic_{g_s}(e_2,f_1)=11-3s^2,\\
        \biRic_{g_s}(e_2,f_2)&=8-\frac{3}{4}s^2,\quad 
         \biRic_{g_s}(f_1,f_2)=11-3s^2.
    \end{split}
\end{equation*}
Therefore, when $s^2\in (16/9,11/3]$, we have $\biRic \geq 0$.

Recall that
\begin{equation*}
        1\operatorname{-triRic}(x;y,z)=\Ric(x,x)+\biRic(y,z)-\sec(x,y)-\sec(x,z),
    \end{equation*}
    and note that $1\operatorname{-triRic}(x;y,z)$ is symmetric in $(x,y,z)$.

Direct calculation yields when $s^2\in (16/9,11/3]$,

\begin{align*}
    1\operatorname{-triRic}(v;\ove_1,e_2)&=1\operatorname{-triRic}(v;\ove_1,f_2) = 11-\frac{3}{2}s^2>0,\\
    1\operatorname{-triRic}(v;\ove_1,f_1)&=8+\frac{3}{4}s^2>0,\\
    1\operatorname{-triRic}(v;e_2,f_1)&=1\operatorname{-triRic}(v;f_1,f_2)=11-\frac{3}{2}s^2+(1-\frac{3}{4}s^2)\sin^2\alpha>0,\\
    1\operatorname{-triRic}(v;e_2,f_2)&=8+\frac{3}{4}s^2+(4-3s^2)\sin^2\alpha>0;    
\end{align*}

\begin{align*}
    1\operatorname{-triRic}(\ove_1;e_2,f_1)&=1\operatorname{-triRic}(\ove_1;f_1,f_2)=12-\frac{9}{4}s^2+(-1+\frac{3}{4}s^2)\sin^2\alpha>0,\\
    1\operatorname{-triRic}(\ove_1;e_2,f_2)&= 12-\frac{9}{4}s^2+(3s^2-4)\sin^2\alpha>0;
\end{align*}

\begin{align*}
    1\operatorname{-triRic}(e_2;f_1,f_2)&=12-\frac{9}{4}s^2>0.
\end{align*}
In conclusion, when $16/9< s^2<11/3$, $(\bS^5, g_{s})$ has positive $\Ric$, $\biRic$ and $1\operatorname{-triRic}$ curvatures with some negative sectional curvature.

Moreover, when $10/3< s^2$,
\begin{equation*}
    \sec(\ove_1,f_1)+\sec(v,f_1)+\sec(e_2,f_1)=-\frac{3}{2}s^2+5<0,
\end{equation*}
which implies that $\Ric_3<0$ somewhere. Thus, $(\bS^5, g_{s})$ with $s\in(10/3,11/3)$ is a family of desired ambient manifolds. By Theorem \ref{maintheorem1introduction}, such ambient manifolds do not admit any two-sided complete noncompact minimal hypersurfaces. 

To end the appendix, we would like to mention that when $s^2\in (\frac{11}{3},4)$, $\Ric_{g_s}\geq 3\left(2-\frac{s^2}{2}\right)>0$, while $\biRic_{g_s}(e_1,e_2)=11-3s^2<0$ and $\inf \sec_{g_s}=4-\frac{9}{4}s^2<0$. This means that positive Ric curvature ($C_1$ curvature) does not imply positive $\biRic$ ($C_2$ curvature) in general. It is also interesting to understand whether there exists a two-sided complete noncompact stable minimal hypersurface in such manifolds $(\bS^5,g_s)$.


\bibliography{mybib}
\bibliographystyle{alpha}

\end{document}

%% file: header.tex
\usepackage{amsmath}
\usepackage{amssymb}
\usepackage{amsthm}
\usepackage[msc-links,abbrev]{amsrefs}
\usepackage{hyperref}
\usepackage[noabbrev,capitalize]{cleveref}
\usepackage{mathrsfs}
\usepackage{mathtools}
\usepackage{slashed}
\usepackage{tikz-cd}
\usepackage{enumitem}

\setenumerate{label=(\roman*)}

\DeclareMathOperator{\tr}{tr}
\DeclareMathOperator{\Vol}{Vol}

\DeclareMathOperator{\inj}{inj}

\DeclareMathOperator{\Ric}{Ric}
\DeclareMathOperator{\biRic}{biRic}

\DeclareMathOperator{\R}{R}

\DeclareMathOperator{\Ad}{Ad}

\DeclareMathOperator{\vspan}{span}

\DeclareMathOperator{\Lip}{Lip}

\newcommand{\oB}{\overline{B}}
\newcommand{\oE}{\overline{E}}
\newcommand{\oG}{\overline{G}}
\newcommand{\oH}{\overline{H}}

\newcommand{\oM}{\overline{M}}
\newcommand{\oP}{\overline{P}}

\newcommand{\ove}{\overline{e}}

\newcommand{\cg}{\widetilde{g}}

\newcommand{\cE}{\widetilde{E}}

\newcommand{\lp}{\langle}
\newcommand{\rp}{\rangle}

\newcommand{\mA}{\mathcal{A}}

\newcommand{\mH}{\mathcal{H}}

\newcommand{\kg}{\mathfrak{g}}
\newcommand{\kh}{\mathfrak{h}}

\newcommand{\kp}{\mathfrak{p}}

\newcommand{\bN}{\mathbb{N}}

\newcommand{\bR}{\mathbb{R}}
\newcommand{\bS}{\mathbb{S}}

\def\sideremark#1{\ifvmode\leavevmode\fi\vadjust{\vbox to0pt{\vss
 \hbox to 0pt{\hskip\hsize\hskip1em
 \vbox{\hsize3cm\tiny\raggedright\pretolerance10000
 \noindent #1\hfill}\hss}\vbox to8pt{\vfil}\vss}}}

\newtheorem{theorem}{Theorem}[section]
\newtheorem{proposition}[theorem]{Proposition}
\newtheorem{lemma}[theorem]{Lemma}

\theoremstyle{definition}
\newtheorem{definition}[theorem]{Definition}

\theoremstyle{remark}
\newtheorem{remark}[theorem]{Remark}

\numberwithin{equation}{section}